\documentclass{amsart}
\usepackage{amsfonts}
\usepackage{amsmath}
\usepackage{amssymb}
\usepackage{amsthm}

\numberwithin{equation}{section}

\DeclareMathOperator{\gw}{\mathsf{g}_{\pm}}
\DeclareMathOperator{\g}{\mathsf{g}}
\DeclareMathOperator{\ord}{ord}
\DeclareMathOperator{\supp}{supp}

\newtheorem{theorem}{Theorem}[section]
\newtheorem{proposition}[theorem]{Proposition}
\newtheorem{lemma}[theorem]{Lemma}
\newtheorem{corollary}[theorem]{Corollary}

\theoremstyle{definition}
\newtheorem{definition}[theorem]{Definition}

\begin{document}

\title[Harborth constant and its weighted analogue]{Some exact values of the Harborth constant and its plus-minus weighted analogue}

\author{Luz E. Marchan \and  Oscar Ordaz \and Dennys Ramos \and Wolfgang A. Schmid}

\address{(L.E.M \& D.R.) Departamento de Matem\'aticas, Decanato de Ciencias y Tecnolog\'{i}as, Universidad Centroccidental Lisandro Alvarado, Barquisimeto, Venezuela}
\address{(O.O.) Escuela de Matem\' aticas y Laboratorio MoST, Centro ISYS, Facultad de Ciencias,
Universidad Central de Venezuela, Ap. 47567, Caracas 1041--A, Venezuela}
\address{(W.A.S.) Universit\'e Paris 13, Sorbonne Paris Cit\'e, LAGA, CNRS, UMR 7539, Universit\'e Paris 8, F-93430, Villetaneuse, France}

\email{luzelimarchan@gmail.com}
\email{oscarordaz55@gmail.com}
\email{ramosdennys@ucla.edu.ve}
\email{schmid@math.univ-paris13.fr}

\thanks{The research of O. Ordaz is supported by the Postgrado de la Facultad de Ciencias de la U.C.V., the CDCH project number 03-8018-2011-1, and the Banco Central de Venezuela; the one of W.A. Schmid by the PHC Amadeus 2012 project number 27155TH and the ANR project Caesar, project number ANR-12-BS01-0011.}



\begin{abstract}
The Harborth constant of a finite abelian group is the smallest integer $\ell$ such that each subset of $G$ of cardinality $\ell$ has a subset of cardinality equal to the exponent of the group whose elements sum to the neutral element of the group. The plus-minus weighted analogue of this constant is defined in the same way except that instead of considering the sum of all elements of the subset one can choose to add either the element or its inverse.
We determine these constants for certain groups, mainly groups that are the direct sum of a cyclic group and a group of order $2$. Moreover, we contrast these results with existing results and conjectures on these problems.
\end{abstract}

\maketitle

\section{Introduction}

We investigate a certain zero-sum constant of finite abelian groups, introduced by Harborth \cite{harborth}, and one of its weighted analogues.
For a finite abelian group $G$, denoted additively, a zero-sum constant of $G$ can be defined as the smallest integer $\ell$ such that each set (or sequence, resp.) of elements of $G$ of cardinality (or length, resp.) $\ell$ has a subset (or subsequence, resp.) whose elements sum to $0$, the neutral element of the group, and that possibly fulfills some additional condition (typically on its size). We refer to the survey article \cite{gaogersurvey} for an overview.

Motivated by a problem on lattice points Harborth considered the constants that arise, for sequences and for sets, when the additional condition on the substructure is that its size is equal to the exponent of the group. For cyclic groups and in the case of sequences this problem had been considered by Erd\H{o}s, Ginzburg, and Ziv \cite{egz} and the resulting constant is thus sometimes called the Erd\H{o}s--Ginzburg-Ziv constant of $G$; see, e.g., \cite{chintamietal} for a recent contribution to this problem.

In the present paper we focus on the constant introduced by Harborth for sets, which we thus call the Harborth constant of $G$; we preserve the classical notation $\g(G)$. The constant $\g(G)$, that is the smallest $\ell$ such that each subset of $G$ of cardinality $\ell$ has a subset of cardinality equal to the exponent of the group whose terms sums to $0$, is only known for very few types of groups. Even in the case of elementary $3$-groups where the problem is particularly popular as it is equivalent to several other well-investigated problems (cap-sets and sets without $3$-term arithmetic progressions) the precise value is only known for rank up to $6$ (see \cite{edeletal} for a detailed overview and \cite{potechin} for the more recent result for rank $6$). Kemnitz \cite{kemnitz} established general bounds for homocyclic groups, from which the exact value for cyclic group follows, namely $\g(C_n)$ equals $n$ or $n+1$, according to $n$ odd or even; note that the constant being $n+1$ in case of even $n$ means that there is no set with the desired property at all, yet for $\ell> n$ the statement is vacuously true. More generally, it is known (see \cite[Lemma 10.1]{gaogersurvey}) that $\g (G) = |G| + 1$ if and only if $G$ is an elementary $2$-group or a cyclic group of even order.

Moreover, Kemnitz showed $\g(C_p^2)=2p-1$ for $p\in \{3,5,7\}$. More recently Gao and Thangadurai \cite{gaothanga} showed $\g(C_p^2)=2p-1$ for prime $p \ge 67$ (this was later refined to $p\ge 47$, see \cite{ggs}) and $\g(C_4^2)= 9$. They then conjectured that $\g(C_n^2)$ equals $2n-1$ or $2n+1$, according to $n$ odd or even, which are the lower bounds obtained by Kemnitz.

Thus, one notices a direct dependence on the parity of the exponent $n$ both for $C_n$ and $C_n^2$, in the latter case at least conjecturally; also the bounds of Kemnitz depend on the parity of the exponent.
One of our results is the exact value of $\g(C_2 \oplus C_{2n})$ for all $n$.

\begin{theorem}
\label{thm_c22n_c}
Let $n\in \mathbb{N}$.
We have
\[
\g(C_2 \oplus C_{2n})=\begin{cases} 2n + 3  &  \text{for $n$ odd }   \\ 2 n + 2 &  \text{for $n$ even} \end{cases}.
\]
\end{theorem}

Again, one observes a direct dependence on the parity of $n$. However, it should be noted that it is of a somewhat different flavor as the exponent of the group, $2n$, is even, regardless.

In addition to these investigations, we also investigate the plus-minus weighted analogue of the Harborth constant.
There are several ways to introduce weights in zero-sum problems. The one we consider here was introduced by Adhikari et al. (see \cite{adetal,adhi0}). Instead of requiring the existence of a zero-sum subsequence or subset one requires only the existence of a plus-minus weighted zero-subsum, that is for each element of the subsequence or subset one is free to choose to add either the element or its inverse (see Section \ref{prel} for a more formal and general definition).
For a recent investigation of the plus-minus weighted analogue of the Erd\H{o}s--Ginzburg--Ziv constant see \cite{adhikari-david}.

We determine the plus-minus weighted Harborth constant, denoted $\gw(G)$, for cyclic groups and groups of the form $C_{2} \oplus C_{2n}$. More specifically we obtain the following results.

\begin{theorem}
\label{thm_cyclic_pm}
Let $n \in \mathbb{N}$.
Then
\[
\gw(C_n) = \begin{cases} n + 1  &  \text{for } n \equiv 2 \pmod{4}  \\ n &  \text{otherwise}  \end{cases}.
\]
\end{theorem}

\begin{theorem}
\label{thm_c22n_w}
Let $n\in \mathbb{N}$.
For $n \ge 3$ we have
\[\gw(C_2 \oplus C_{2n})= 2n +2.\]
Moreover, $ \gw(C_2 \oplus C_{4})= \gw(C_2 \oplus C_{2})= 5$.
\end{theorem}

For cyclic groups one has again a certain dependence on parity, yet somewhat surprisingly this phenomenon does not appear for $C_2 \oplus C_{2n}$, and the result (albeit not its proof) is independent of $n$, except for the expected and known exception for $n=1$ and the phenomenon that for $n=2$ the value is smaller by one than one might expect. In addition, we establish an analogue of \cite[Lemma 10.1]{gaogersurvey} in the presence of plus-minus weights (see Corollary \ref{cor_G+1_weighted}).

\section{Preliminaries and notation}
\label{prel}

We collect some definitions and notations we use frequently.  By $\mathbb{N}$ and $\mathbb{N}_0$ we denote the set of positive and non-negative integers, respectively. For reals $a, b $ we denote by $[a, b] = \{x \in  \mathbb{Z} \colon a  \le x \le b \}$.

Let $G$ be a finite abelian group; we use additive notation. We denote by $C_n$ a cyclic group of order $n$. There are uniquely determined $1< n_1 \mid \dots \mid n_r$ such that $G \cong C_{n_1} \oplus \dots \oplus C_{n_r}$. We call $n_r$ the exponent of $G$, denoted $\exp(G)$; the exponent of a group of order $1$ is $1$.
We call $r$ the rank of $G$; the rank of a group of order $1$ is $0$. The $p$-rank of $G$, for $p$ a prime, is the number of $i$ such that  $p \mid n_i$.

For $G = \oplus_{i=1}^s G_i$, by the projection $\pi_i:G \to G_i$ we mean the group homomorphism $g_1 + \dots + g_s \mapsto g_i$; this depends on the direct sum decomposition not just the groups $G$ and $G_i$, yet at least implicitly it will be clear which decompositions we mean.

A sequence over $G$ is an element of the free abelian monoid over $G$, which we denote multiplicatively.
In other words, for each sequence $S$ over $G$ there exist up to ordering uniquely determined $g_1, \dots, g_{\ell} \in G$ (possibly some of them equal) such that $S= g_1 \dots g_{\ell}$; moreover, there exist unique $v_g \in \mathbb{N}_0$ such that $S= \prod_{g \in G}g^{v_g}$. The neutral element of this monoid, the empty sequence, is denoted by $1$. We denote by $|S|= \ell$ the length of $S$ and by $\sigma(S) = \sum_{i=1}^{\ell
} g_i $ its sum. The set $\{g_1, \dots, g_{\ell}\}$ is called the support of $S$, denoted $\supp(S)$. The sequence $S$ is called square-free if all the $g_i$ are distinct. We say that $T$ is a subsequence of $S$, if $T$ divides  $S$ in the monoid of sequences  that is $T= \prod_{i \in I }g_i  $ for some $I\subset [1, \ell]$. For $T\mid S$ a subsequence we denote by $ST^{-1}$ the sequence fulfilling $T (ST^{-1}) =S $, in other words $ST^{-1} = \prod_{i \in [1,  \ell] \setminus I  }g_i$ if  $T= \prod_{i \in I }g_i  $.

There is an immediate correspondence between squarefree sequences over $G$ and subsets of $G$, in other words we could identify $S$ with $\supp(S)$. While in this  paper we are mainly interested in  squarefree sequences, that is sets, we still use the formalism and language of sequences rather than that of sets. On the one hand, we do so for consistency with other work, yet on the other hand regarding certain aspects there is an actual difference regarding the meaning of standard constructions (see below).

For $W \subset \mathbb{Z}$, we call $ \sum_{i=1}^{\ell}  w_i g_i $ with $w_i \in W$ a $W$-weighted sum of $S$; when arising in this context we refer to $W$ as the set of weights. Moreover, we denote by $\sigma_W( S) = \{   \sum_{i=1}^{\ell}  w_i g_i \colon w_i \in W \}$ the set of all $W$-weighted sums of $S$.

In addition, we need the following notations:
\[
\Sigma_{W}(S)= \left\{   \sum_{i \in I }  w_i g_i \colon w_i \in W ,\, \emptyset \neq I \subset [1,\ell] \right\} = \bigcup_{1 \neq T \mid S} \sigma_W(T)
\]
the set of all $W$-weighted subsums of $S$ as well as the variant
\[
\Sigma_{W}^0(S)= \left\{ \sum_{i \in I }  w_i g_i \colon w_i \in W ,\, I \subset [1,\ell] \right\} = \bigcup_{ T \mid S} \sigma_W(T)
\]
where the empty subsum is also permitted.
Note that we always have that $\Sigma_{W}^0(S) = \Sigma_{W}(S) \cup \{0\}$, yet not always  $\Sigma_{W}^0(S) \setminus \{0\} = \Sigma_{W}(S)$.

Moreover, we denote
\[
\Sigma_{W,k}(S)= \left\{ \sum_{i \in I }  w_i g_i \colon w_i \in W ,\, \emptyset \neq I \subset [1,\ell] , \, |I|=k \right\} = \bigcup_{1 \neq T \mid S,\, |T|=k} \sigma_W(T)\]
the set of all $W$-weighted subsums of $S$ of length $k$.

For $W=\{1\}$, one recovers the usual notions in the non-weighted case, and we drop the subscript $W$ in this case, except for the fact that strictly speaking
$\sigma_{\{1\}}(S)$ is not $\sigma(S)$ but rather the set containing $\sigma(S)$. We continue to consider $\sigma(S)$ as an element of $G$ rather than as a singleton  set containing this element. Moreover, we use the symbol $\pm$ instead of $W$ for $W= \{+1, -1\}$ and speak of plus-minus weighted sums.

For a map $ \varphi \colon G \to G'$, where $G'$ denotes an abelian group, there is a unique continuation of $\varphi$ to a monoid homomorphism from the monoid of sequences over $G$ to the monoid of sequences over $G'$, which we continue to denote by $\varphi$. More explicitly,   $\varphi(S) = \varphi(g_1) \dots \varphi(g_{\ell})$. We point out that if $\varphi$ is not injective, then the image under $\varphi$ of a squarefree sequence might not be a squarefree sequence, yet we always have $|S|= |\varphi(S)|$. Here, the situation would be different if we consider $S$ as a set, and this is a main reason why we prefer to work with sequences.
If $\varphi$ is not only a map, but in fact a group homomorphism, then $\varphi(\sigma (S))=\sigma (\varphi (S))$, and likewise for $\sigma_W$, $\Sigma_W^{0}$, $\Sigma_W$ and $\Sigma_{W,k}$.

Let $A, B\subset G  $ then $A+B = \{a+b \colon a\in A, \, b \in B\}$ denotes the sum of the sets $A$ and $B$, and  $A \hat{+}  B = \{a+b \colon a\in A, \, b \in B , \, a \neq b\}$ the restricted sum of $A$ and $B$. For $g \in G$, we write $g + A$ instead of $\{g\} + A$.
For $k \in \mathbb{Z}$, we denote by $k \cdot A = \{k a \colon a \in A\}$ the dilation of $A$ by $k$, not the $k$-fold sum of $A$ with itself. We write $-A$ instead of $(-1) \cdot A$.
Also, for $S= g_1 \dots g_{\ell}$ a sequence we use the notations $g+ S$ to denote the sequence $(g + g_1) \dots (g + g_{\ell})$ and $-S$ to denote the sequence $(-g_1) \dots (-g_{\ell})$.

\section{Main definitions and auxiliary results}

The focus of this paper is the investigation of the Harborth constant and its plus-minus weighted analogue. We recall its definition, for arbitrary set of weights, as well as related definitions in a formal way.

\begin{definition}
Let $G$ be a finite abelian group. Let $W \subset \mathbb{Z}$.
The $W$-weighted  Harborth constant of $G$, denoted by $\g_W(G)$, is the smallest $\ell \in \mathbb{N}$ such that for each squarefree sequence over $G$ with $|S| \ge \ell$ we have $0 \in \Sigma_{W, \exp(G)}(S)$.
\end{definition}
The (classical) Harborth constant is the special case $W= \{1\}$, i.e., without weights; the plus-minus weighted Harborth constant is the special case $W= \{+1,-1\}$; we denote them by $\g (G)$ and $\gw(G)$, respectively.

We also use the $W$-weighted Olson constant, defined in the same way, except that the condition is $0 \in \Sigma_{W}(S)$, that is we do not impose any condition on the length of the weighted zero-subsum (except for it being non-empty).

While we do not use them in this paper, but as we mentioned them in the Introduction, we recall that the $W$-weighted Erd\H{o}s--Ginzburg--Ziv constant, denoted $\mathsf{s}_W(G)$, is the constant one gets when replacing `squarefree sequence' by `sequence' in the definition of the $W$-weighted Harborth constant. Likewise, the analogue for sequence of the   $W$-weighted Olson constant is the $W$-weighted Davenport constant. Observe that in case $\{+1,-1\} \subset W$, the $W$-weighted Olson constant and the $W$-weighted Davenport constant are equal. Yet, this is in general not true for $\g_W(G)$ and $\mathsf{s}_W(G)$.

We start by establishing a simple general lemma on the behavior of the $W$-weighted Harborth constant with respect to direct sum decompositions of the group.

\begin{lemma}
\label{lem_lb_gen}
Let $G$ be a finite abelian group. Let $W \subset \mathbb{Z}$ be a set of weights. If $G= H\oplus K$ with $\exp(H) \mid \exp(K)$, then
\(
\g_W (G) \ge \mathsf{O}_W (H) +  \g_W (K) -1.
\)
\end{lemma}
\begin{proof}
Let $B$ be a square-free sequence over $K$ of length $\g_W (K) -1$ that does not contain a $W$-weighted zero-subsum of length $\exp(K)$. Furthermore, let $A$ be a square-free sequence over $H$ of length $\mathsf{O}_W (H) -1$ that does not contain a nonempty $W$-weighted zero-subsum. Since  $\exp(H) \mid \exp(K)$, we have $\exp(K)= \exp(G)$.
We assert that the sequence $AB$ cannot have a zero-subsum of length $\exp(G)$. On the one hand, such a subsum cannot contain an elements from $A$, since those would yield a nonempty $W$-weighted zero-subsum of $A$. On the other hand, it cannot contain no element of $A$, since then it would yield a $W$-weighted zero-subsum of $B$ of length $\exp(G)= \exp(K)$. Therefore, it $AB$ cannot have such a subsum and $\gw_W(G)>|AB|=  (\mathsf{O}_W (H) - 1 ) +  (\g_W (K) -1) $, establishing the claim.
\end{proof}

We need the following (well-known) lemma; for lack of a suitable reference for the second part, we include the short proof.

\begin{lemma}
\label{lem_fullgroup}
Let $G$ be a finite abelian group, and let $t$ denote the $2$-rank of $G$. Let $A,B \subset G$ nonempty subsets.
\begin{enumerate}
\item Suppose $|A| + |B| \ge  |G| + 1 $. Then $A+B = G$.
\item Suppose $|A| + |B| \ge  |G |+ 1 + 2^t$. Then $A \hat{+} B = G$.
\end{enumerate}
\end{lemma}
\begin{proof}
1. Let $g\in G$.  Since $|A| + |B| \ge |G| +1$ and $|g-A| =|A|$, it follows that $(g-A ) \cap B$ is nonempty. That is there are $a \in A$ and $b \in B $ such that $g-a= b$, that is $g=a+b$, so $g \in A+B$. Since $g$ was arbitrary, we have $A+B= G$.

\noindent
2. Let $g\in G$. Since $|A| + |B| \ge  |G| + 1 + 2^t$, it follows that $|(g-A) \cap B| \ge 1 + 2^t$. Thus there are (distinct) $a_i \in A$ and (distinct) $b_i \in B$ for $i \in [1, 2^t + 1]$ such that $g - a_i = b_i$. If for some $i$ we have $a_i \neq b_i$, then we get as above $g \in A \hat{+}B$ and the claim.
So, assume for each $i$, we have  $a_i = b_i$,  which implies $g= 2a_i$. However, for fixed $g$ the solutions of the equation $g = 2 X$, in other words the pre-image of $g$ under the group endomorphism `multiplication by $2$', are contained in a co-set to the subgroup of $G$ formed by the elements of order at most $2$.  Yet, this subgroups and thus the coset has cardinality $2^t$. So, it is impossible that $a_i = b_i$ and thus $g = 2a_i$ holds true for each  $a_i$, establishing the claim.
\end{proof}

We give a way to express, for $W=\{+1,-1\}$, the set of $W$-weighted sums of a  sequence in terms of notions not involving weights.

\begin{lemma}
\label{lem_weightsumset}
Let $G$ be a finite abelian group and let $S$ be a sequence over $G$. Then $\sigma_{\pm}(S) = - \sigma ( S) + 2 \cdot \Sigma^0(S)$.
In particular, if $|G|$ is odd, then $|\sigma_{\pm}(S)| = |\Sigma^0(S)| \ge 1 + |\supp(S) \setminus \{0\} |$.
\end{lemma}
\begin{proof}
Let $S= g_1 \dots g_k$. We have that  $g \in \sigma_{\pm}(S) $ if and only if $g = \sum_{i=1}^k \varepsilon_i g_i$ with $\varepsilon_i \in \{+ 1, - 1\}$. This is equivalent to
$g = -\sum_{i=1}^k  g_i + \sum_{i=1}^k \delta_i g_i$ with $\delta_i \in \{0, 2\}$. Now, $h= \sum_{i=1}^k \delta_i g_i$ with $\delta_i \in \{0, 2\}$ is just another way of saying $h \in 2 \cdot \Sigma^0 (S)$. The claim is established.
\end{proof}

For the clarity of the exposition we stated this lemma for plus-minus weighted sequences only, however the exact same argument allows to show that for distinct $v,w \in \mathbb{Z}$ one has $\sigma_{\{v,w \}}(S) =  v\sigma ( S) + (w-v) \cdot \Sigma^0(S)$, and the additional assertions with the condition $|G|$ odd replaced by $w-v$  co-prime to $|G|$.

\section{Cyclic groups}

In this section we proof our results for cyclic groups.
As mentioned in the Introduction, the value of $\g(C_n)$ is known for cyclic groups, more precisely, for $n \in \mathbb{N}$, we have
\begin{equation}
\label{eq_gcyc}
\g(C_n) = \begin{cases} n + 1  &  \text{for $n$ even}  \\ n &  \text{for $n$ odd}  \end{cases}.
\end{equation}

Now, we proceed to prove Theorem \ref{thm_cyclic_pm}, showing that in the presence of weights the situation is somewhat different, yet there is still a direct dependence on the `parity' of $n$, or to be precise on a congruence condition modulo $4$.

\begin{proof}[Proof of Theorem \ref{thm_cyclic_pm}]
We need to show that, for $n \in \mathbb{N}$,
$\gw(C_n)$ is equal to  $n + 1$  for  $n \equiv 2 \pmod{4}$ and equal to  $n$ otherwise.

In both cases we have $\gw (C_n) \ge n$.
For $n$ odd, the claim follows by noting that  $\gw (C_n)  \le \g(C_n)= n$, the last equality by \eqref{eq_gcyc}.

Suppose $n$ is even.
Let $C_n = \langle e \rangle$ and let $T = \prod_{i=0}^{n-1} (i e)$, which is the only squarefree sequence over $C_n$ of length $n$.
Now, we note that
\[
\sum_{i=0}^{n/2} i - \sum_{i=n/2 + 1}^{n-1} i  = 0 + \frac{n}{2} + \sum_{i=1}^{n/2-1} i -  \left( i+\frac{n}{2} \right) = \frac{n}{2} - \left( \frac{n}{2} - 1 \right) \frac{n}{2} = n  \left( 1 - \frac{n}{4} \right).
\]
Thus, for $n\equiv 0 \pmod{4}$, we have that the sum above is congruent to $0$ modulo $n$. Consequently,  $0 \in \sigma_{\pm} ( T)$. We have thus shown that for $n\equiv 0 \pmod{4}$, one has $\gw (C_n)  \le  n$.

It remains to show that  $0 \notin \sigma_{\pm} (T)$ for $n\equiv 2 \pmod{4}$. We note that $-\sigma(T) = -(n/2)(n-1)e = (n/2)e$. By Lemma \ref{lem_weightsumset} we have $\sigma_{\pm} (T) = -\sigma(T) + 2 \cdot \Sigma^{0}(T)$.
Since by assumption $n/2$ is odd, of course, $ (n- 2)/2 $ is even, and thus
\[
-\sigma(T) + 2 \cdot \Sigma^{0}(T) =  \left(e + \frac{n - 2}{4}   2 e \right)  + 2 \cdot \Sigma^0(T)   \subset e +  2 \cdot \langle  e \rangle .
\]
Since $2  \cdot \langle  e \rangle =  \langle 2 e \rangle$ is a proper subgroup of $\langle  e \rangle$, it follows that $0 \notin \sigma_{\pm} (T)$.
\end{proof}

Next, we determine all groups for which $\gw(G)> |G|$. The argument uses the preceding result as well as the analogous result in the classical case, which we recalled in the Introduction.

\begin{corollary}
\label{cor_G+1_weighted}
Let $G$ be a finite abelian group. Then $\gw (G) =  |G|+1 $ if and only if $G$ is an elementary $2$-group or a cyclic group of order congruent $2$ modulo $4$.
\end{corollary}
\begin{proof}
We recall from the Introduction that it is known that $\g (G) = |G| + 1$ if and only if $G$ is an elementary $2$-group or a cyclic group of even order.
Since $\gw(G) \le \g (G)$, we thus have that  $\gw (G) \le  |G|$, for all other groups.
For $G$ an elementary $2$-group it is clear that $\gw(G) = \g(G)$, and $\gw(G)= |G| + 1$. It remains to consider the case that $G$ is a cyclic group of even order.
In this case Theorem \ref{thm_cyclic_pm} gives that $\gw(G)$ is equal to  $|G| $ for $|G|$ congruent to $0$ modulo $4$ and to $|G|+ 1$ for $|G|$ congruent to $2$ modulo $4$. This completes the argument.
\end{proof}

\section{Groups of the form $C_2 \oplus C_{2n}$}

In this section we determine $\g (C_2 \oplus C_{2n})$ and $\gw(C_2 \oplus C_{2n})$ for each $n$.
We first make explicit the lower bounds that follow from the results for cyclic groups established in the preceding section in combination with Lemma \ref{lem_lb_gen}. We have, for $n \in \mathbb{N}$,
\begin{equation}
\label{eq_c}
\g(C_2 \oplus C_{2n}) \ge 2n+2 \quad \text{and}
\end{equation}
\begin{equation}
\label{eq_w}
\gw (C_2 \oplus C_{2n}) \ge \begin{cases} 2n + 2  &  \text{for $n$ odd }   \\  2 n + 1 &  \text{for $n$ even}
\end{cases}.
\end{equation}
This follows directly by \eqref{eq_gcyc}, Theorem \ref{thm_cyclic_pm}, and Lemma \ref{lem_lb_gen}, using the simple fact that $\mathsf{O}(C_2)= \mathsf{O}_{\pm}(C_2)= 2$.

As can be seen by comparing these bounds with Theorems \ref{thm_c22n_c} and \ref{thm_c22n_w}  we show that sometimes but not always these bounds are sharp. We start by showing that for the plus-minus weighted version one in fact has $\gw(C_2 \oplus C_{2n}) \ge 2n+2$ for each $n \ge 3$ not just for odd ones.

\begin{lemma}
\label{lem_lb_w}
Let $n \in \mathbb{N}$ with $n \ge 3$. Then
$\gw(C_2 \oplus C_{2n}) \ge 2n+2$.
\end{lemma}
\begin{proof}
If $n$ is odd, this is just \eqref{eq_w}. Thus assume $n$ is even. Let $C_2 \oplus C_{2n} = \langle e_1 \rangle \oplus \langle e_2 \rangle $ with $\ord (e_1)= 2$ and $\ord (e_2)= 2n$. We consider the sequence $S=TR$ with $T=\prod_{i=1}^{2n-2}(i e_2)$ and $R= e_1(e_1 + 2e_2) (e_1 + 4e_2)$.

Since $n\ge 3$ we have $2n \ge 4$ and $R$ is square-free, and also $S$ is square-free and a sequence of length $2n+1$. It thus suffices to show that $0$ is not a plus-minus weighted subsum of length $2n$ of $S$. We observe that
\[
\sigma(T) =  \frac{(2n-2)(2n-1)}{2} e_2 = (n+1)e_2.
\]
It follows that $\sigma_{\pm}(T) = -\sigma(T) + 2 \cdot \Sigma^0(T)  \subset e_2 + \langle 2e_2 \rangle$; note that $n$ is even by assumption and thus $ne \in  \langle 2e_2 \rangle$.

Let $S' \mid S$ be a subsequence of length $2n$. It is clear that $S'$ cannot be a subsequence of $T$, since $|T|= 2n-2$. Moreover, if $0 \in \sigma_{\pm}(S')$, then $R'$ the subsequence of elements in $S'$ from $R$ has to have even length,  otherwise the projection to $\langle e_1 \rangle$ of a  weighted sum of $S'$ cannot be $0$. Consequently, $|R'|=2$, since it is non-zero, even and at most $3$. Therefore, $S' = R'T$. However, $\sigma_{\pm}(R') \subset \langle 2e_2 \rangle $ while  $\sigma_{\pm}(T) \subset e_2 + \langle 2e_2 \rangle$, so $0 \notin   \sigma_{\pm} (R')  +  \sigma_{\pm} (T)  = \sigma_{\pm} (R'T) $.
\end{proof}

We now establish an improvement for  the lower bound for $\g(C_2 \oplus C_{2n})$ for $n$ odd.

\begin{lemma}
\label{lem_lb_c}
Let $n \in \mathbb{N}$  odd. Then $\g(C_2 \oplus C_{2n}) \ge 2n+3$.
\end{lemma}
\begin{proof}
First, we observe that the result holds for $n=1$ as then the group is an elementary $2$-group (see the Introduction).  Thus, assume $n \neq 1$.
Since $n$ is odd, we have $C_2 \oplus C_{2n}= \langle f_1\rangle  \oplus \langle f_2 \rangle \oplus \langle e \rangle $ with $\ord (f_i) = 2 $ and $ \ord (e)= n$.
Let $\pi_1$ denote the projection to $\langle f_1\rangle  \oplus \langle f_2 \rangle$ and let $\pi_2$ denote the one to $ \langle e \rangle$.

We construct a squarefree sequence $A$ of length $2n+2$ without zero-sum subsequence of length $2n$.
Let $B = 0 \prod_{i=1}^{(n-1)/2} (ie)$ and let
\[
A = B (f_1 +f_2  + B) (f_1 - B) (f_2 - B).
\]
Then $|A| = 4(n+1)/2 = 2n+2$ and $A$ is squarefree. Moreover, $\sigma(A)= 0$, which can be seen directly by considering $\pi_1(A)$ and $\pi_2(A)$.
Suppose $A$ has a zero-sum subsequence $S$ of length $2n$. Then $AS^{-1}$ is also a zero-sum subsequence of $A$, and it has length two.
Yet, $A$ cannot have a zero-sum subsequence of length two, since there are no two elements $gh\mid A$ such that $g+h = 0$.
\end{proof}

Next, we establish the upper bound for $\g(C_2 \oplus C_{2n})$ for even $n$. This result is also used in the proof of Theorem \ref{thm_c22n_w}, thus we formulate it separately.

\begin{proposition}
\label{prop_22n_pair}
Let $n \in \mathbb{N}$  even. Then
$\g(C_2 \oplus C_{2n}) \le  2n+2$.
\end{proposition}
\begin{proof}
Let $C_2 \oplus C_{2n} = \langle e_1 \rangle \oplus \langle e_2 \rangle $ with $\ord (e_1)= 2$ and $\ord (e_2)= 2n$; let $\pi_1$ and $\pi_2$ denote the projections on $\langle e_1 \rangle$ and $\langle e_2 \rangle$, respectively.

Let $A$ be a square-free sequence over  $C_2 \oplus C_{2n}$ of length $2n+2$ and write $A=A_0A_1$ where $\pi_1(g)$ is $0$ and $e_1$ for $g \mid A_0$ and $g \mid A_1$, respectively.
We distinguish the two cases $\pi_1(\sigma(A))\neq 0$ and $\pi_1(\sigma(A))= 0$.

Assume $\pi_1(\sigma(A))\neq 0$, that is $\pi_1(\sigma(A))= e_1$. Since $\pi_2(A_i)$ is squarefree and since we have $|\pi_2(A_0)| + |\pi_2(A_1)| = |A| = 2n + 2 > 2n$, it follows by Lemma \ref{lem_fullgroup} that $\supp(\pi_2(A_0)) + \supp(\pi_2(A_1)) = \langle e_2\rangle$, that is each element of $ \langle e_2 \rangle$ is the sum of an element appearing in $\pi_2(A_0)$ and an element appearing in $\pi_2(A_1)$. Thus, let $g_i \mid A_i$ such that $\pi_2(g_0) + \pi_2(g_1) = \pi_2 (\sigma (A))$.
Set $S= A(g_0g_1)^{-1}$. Then $\pi_2 (\sigma (S)) = \pi_2 (\sigma (A)) -  (\pi_2(g_0) + \pi_2(g_1)) = 0 $.
Moreover, $\pi_1 (\sigma (S)) = \pi_1 (\sigma (A)) -  (\pi_1(g_0) + \pi_1(g_1)) = e_1 - e_1 = 0$.
Thus $\sigma(S)= 0$. Since $|S| = (2n+2)- 2 = 2n$, the argument is complete in this case.

Assume $\pi_1(\sigma(A))= 0$. Let $\{x,y\}= \{0,1\}$ such that $|A_x | \ge |A_y|$. We have $|A_x|\ge (2n + 2)/2 = n+1$. Yet, note that if $|A_x| = n+1$, then also $|A_y| = n+1$, so $|A_1|=n+1$; this contradicts $\pi_1(\sigma(A))= 0$ since $\pi_1(\sigma(A))= |A_1|e_1$ and $n+1$ is odd as $n$ is even by assumption (this is the only place where we use this assumption). Consequently, $|A_x| \ge n+2$.

By Lemma \ref{lem_fullgroup} we have $\supp(\pi_2 (A_x)) \hat{+} \supp(\pi_2 (A_x)) = \langle e_2 \rangle$. Let $gh\mid A_x$ such that $\pi_2(g) + \pi_2(h) = \pi_2 (\sigma (A))$, which exist by the just made observation.
Set $S= A(gh)^{-1}$. Then $\pi_2 (\sigma (S)) = \pi_2 (\sigma (A)) -  (\pi_2(g) + \pi_2(h)) = 0 $.
Moreover, $\pi_1 (\sigma (S)) = \pi_1 (\sigma (A)) -  (\pi_1(g) + \pi_1(h)) =0 -( e_1 + e_1) = 0$.
Thus $\sigma(S)= 0$. Since $|S| = (2n+2)- 2 = 2n$, the argument is again complete in this case.
\end{proof}

The following technical result is used in the proofs of both Theorem \ref{thm_c22n_c} and Theorem \ref{thm_c22n_w}.
We only need it for $n$ odd, when the group is isomorphic to $C_2 \oplus C_{2n}$, yet as it does not cause any complication we state it for general $n$.

\begin{proposition}
\label{prop_nonzero}
Let $n \in \mathbb{N}$. Let  $\pi_1: C_2 \oplus C_2 \oplus C_n \to C_2 \oplus C_2$.
Let $A$ be a squarefree sequence over $ C_2 \oplus C_2 \oplus C_n$ of length $2n+2$. If $\sigma ( \pi_1 (A)) \neq 0$, then $A$ has a zero-sum subsequence of length $2n$.
\end{proposition}
\begin{proof}
For $n=1$,  the claim is vacuously true as the only squarefree sequence of length $4$ over $ C_2 \oplus C_2 \oplus C_1$ has sum $0$. Thus assume $n \neq 1$.

We write $ C_2 \oplus C_2 \oplus C_n =   \langle f_1\rangle  \oplus \langle f_2 \rangle \oplus \langle e \rangle $ with $\ord (f_i) = 2 $ and $ \ord (e)= n$.
In addition to $\pi_1$, let $\pi_2$ denote the projection  $  C_2 \oplus C_2 \oplus C_n \to  \langle e \rangle$.
We denote the elements of  $\langle f_1\rangle  \oplus \langle f_2 \rangle$ by  $\{0,a,b,c\}$ where $0$ is of course the neutral element; note that in any case $b+c = a$. Now we write $A=A_0A_a A_b A_c$ where $\pi_1(g)=x$  for $g \mid A_x$. As in the claim of the result, suppose  $\sigma ( \pi_1 (A)) \neq 0$, say,  $\sigma ( \pi_1 (A)) = a$.

Let $(B_1,B_2)$ equal $(A_0,A_a)$ or $(A_b,A_c)$ such that $|B_1| + |B_2| =  \max \{ |A_0|+|A_a |,     |A_b|+|A_c| \}$. Then $|B_1|+|B_2| \ge (2n+2)/2 = n+1$. Thus by Lemma \ref{lem_fullgroup}  $ \supp (\pi_2(B_1)) + \supp(\pi_2(B_1)) = \langle e \rangle $; note that $\pi_2(B_i)$ is squarefree and thus $|\supp (\pi_2(B_i))|= |B_i|$.  Let $g_i \mid B_i$ such that $\pi_2(g_1) + \pi_2( g_2 ) = \pi_2(\sigma(A)) $. Then  $\pi_2( \sigma(  A(g_1g_2)^{-1}) ) = 0$ and $\pi_1( \sigma(  A(g_1g_2)^{-1}) ) $ also equals $0$ as $(\pi_1(g_1),\pi_1(g_1)) $ equals $(0,a)$ or $(b,c)$.  Thus we have a zero-sum subsequence of $A$ of length $2n$.
\end{proof}

Before proceeding to prove the main results, we consider a special case.

\begin{lemma}
\label{lem_c24}
$\gw(C_2 \oplus C_4) = 5$.
\end{lemma}
\begin{proof}
We have $\gw(C_2 \oplus C_{2n}) \ge 5$ by \eqref{eq_w}. We establish that $5$ is also an upper bound.
Let $C_2 \oplus C_{4} = \langle e_1 \rangle \oplus \langle e_2 \rangle $ with $\ord (e_1)= 2$ and $\ord (e_2)= 4$; let $\pi_1$ and $\pi_2$ denote the projections on $\langle e_1 \rangle$ and $\langle e_2 \rangle$, respectively.

Let $A$ be a square-free sequence over  $C_2 \oplus C_{4}$ of length $5$ and write $A=A_0A_1$ where $\pi_1(g)$ is $0$ and $e_1$ for $g \mid A_0$ and $g \mid A_1$, respectively.   Let $\{x,y\} = \{0,1\}$ such that $|A_x| \ge |A_y|$. If we have $|A_x|\ge 4$, then by Theorem \ref{thm_cyclic_pm} we get that $\pi_2(A_x)$, which is a squarefree sequence over $C_4$, has a plus-minus weighted zero-subsum of length $4$. Since for each plus-minus weighted subsum of length $4$ of $A_x$ the value under $\pi_1$ is $0$ as it is always $4x=0$ the choice of sign being irrelevant for an element of order $2$, the  plus-minus weighted zero-subsum of lengths $4$ of $\pi_2(A_x)$, yields in fact a plus-minus weighted zero-subsum of length $4$ of $A_x$.

Thus, it remains to consider $|A_x|= 3$ and $|A_y|=2$. We note that $\sigma_{\pm} (A_y)$ contains a non-zero element of $ \langle e_2 \rangle $. Yet $\Sigma_{\pm,2}(A_x)$ contains $\langle e_2 \rangle \setminus \{0\}$, which can for example be seen by considering the four possible cases. Thus, we get a plus-minus weighted zero-subsum of length $4$.
\end{proof}

Now, we give the proofs of the two main results  Theorem \ref{thm_c22n_w} and Theorem \ref{thm_c22n_c}.

\begin{proof}[Proof of Theorem \ref{thm_c22n_w}]
For $n=1$ we have $\gw(C_2 \oplus C_2)=5$ by Corollary \ref{cor_G+1_weighted}, and for $n=2$, we have $\gw(C_2 \oplus C_{2n}) = 5$ by Lemma \ref{lem_c24}.
For $n\ge 3$, we have $\gw(C_2 \oplus C_{2n}) \ge 2n +2$ by Lemma \ref{lem_lb_w}.
It remains to show  $\gw(C_2 \oplus C_{2n}) \le 2n +2$ for $n \ge 3$. For even $n$ this follows by Proposition \ref{prop_22n_pair} as $\gw(C_2 \oplus C_{2n}) \le \g(C_2 \oplus C_{2n})$. Thus, suppose $n \ge 3$ is odd.

Since $n$ is odd, we have $C_2 \oplus C_{2n}= \langle f_1\rangle  \oplus \langle f_2 \rangle \oplus \langle e \rangle $ with $\ord (f_i) = 2 $ and $ \ord (e)= n$.
Let $\pi_1$ denote the projection to $\langle f_1\rangle  \oplus \langle f_2 \rangle$ and let $\pi_2$ denote the one to $ \langle e \rangle$.
 Let $A$ be a square-free sequence over  $C_2 \oplus C_{2n}$ of length $2n+2$. Denote $\langle f_1\rangle  \oplus \langle f_2 \rangle = \{0,a,b,c\}$; note that $b+c = a$. Now we write $A=A_0A_a A_b A_c$ where $\pi_1(g)=x$  for $g \mid A_x$.
We distinguish the two cases $\pi_1(\sigma(A))\neq 0$ and $\pi_1(\sigma(A))= 0$.

Assume  $\pi_1(\sigma(A)) \neq 0$. The existence of a zero-sum subsequence, and thus in particular a plus-minus weighted zero-subsum, follows directly by Proposition \ref{prop_nonzero}.

Assume  $\pi_1(\sigma(A)) = 0$. Let $x \in \{0,a,b,c\}$ such that $|A_x|$ is maximal. Then $|A_x| \ge (2n + 2) / 4 = (n+1)/2$. Suppose there exist $gh|A_x$ such that $\pi_2 (g) + \pi_2 (h)=  \pi_2 (\sigma(A))$. Then, as $\pi_1 (g) = \pi_1 (h)$, we get that $A(gh)^{-1}$ is a zero-sum sequence of length $2n$, and we are done.

Thus assume such $gh$ does not exist. By Lemma \ref{lem_fullgroup} this implies that $|A_x|=(n+1)/2$, and thus $|A_y|=(n+1)/2$ for each  $y \in \{0,a,b,c\}$.
Let now $gh|A_0$ arbitrary. We note that $\pi_1(\sigma_{\pm} (A (gh)^{-1})) = \{0\}$. To complete the proof it thus suffices to show that $0 \in \pi_2(\sigma_{\pm} (A (gh)^{-1}))$. In fact, we now show  $\pi_2(\sigma_{\pm} (A (gh)^{-1})) = \langle f \rangle$. We observe
\(|\sigma_{\pm} (\pi_2 ( A(gh)^{-1} ))  |  \ge | \sigma_{\pm} (\pi_2 ( A_aA_b))  | = | \sigma_{\pm} (\pi_2 ( A_a)) + \sigma_{\pm} (\pi_2 ( A_b))  |.
\)
 We have  $|\sigma_{\pm} (\pi_2 ( A_a))| = |\Sigma^0 (\pi_2 ( A_a))| \ge  |\pi_2 ( A_a)| = |A_a| = (n+1)/2 $, where we used Lemma \ref{lem_weightsumset} for the first equality; and the same for $b$ instead of $a$. Thus, by Lemma \ref{lem_fullgroup} $ \sigma_{\pm} (\pi_2 ( A_a)) + \sigma_{\pm} (\pi_2 ( A_b))   = \langle f \rangle$ and  thus  $| \sigma_{\pm} (\pi_2 ( A(gh)^{-1} )) | =  n $ establishing the claim.
\end{proof}

\begin{proof}[Proof of Theorem \ref{thm_c22n_c}]
First, we observe that the result holds for $n=1$, since the group is an elementary $2$-group (see the Introduction).  Thus, assume $n \neq 1$.
For even $n$, we have $\g(C_2  \oplus C_{2n}) \ge 2n + 2$ by \eqref{eq_c} and $\g(C_2  \oplus C_{2n}) \le 2n + 2$ by Proposition \ref{prop_22n_pair}.

Now, assume $n$ is odd. We have $\g(C_2  \oplus C_{2n}) \ge 2n + 3$ by Lemma \ref{lem_lb_c}. It remains to show $\g(C_2  \oplus C_{2n}) \le 2n + 3$.
Since $n$ is odd, we have $C_2 \oplus C_{2n}= \langle f_1\rangle  \oplus \langle f_2 \rangle \oplus \langle e \rangle $ with $\ord (f_i) = 2 $ and $ \ord (e)= n$.
Let $\pi_1$ denote the projection to $\langle f_1\rangle  \oplus \langle f_2 \rangle$ and let $\pi_2$ denote the one to $ \langle e \rangle$.
Let $A$ be a square-free sequence over  $C_2 \oplus C_{2n}$ of length $2n+3$.
We have to show that it has a zero-sum subsequence of length $2n$. Let  $g \mid A$  such that $  \pi_1( g ) \neq \pi_1(\sigma( A ))$; note such an element exists since at most $n$ of the elements of $A$, and in fact $C_2 \oplus C_{2n}$, can have the same value under $\pi_1$.

Let   $A'= Ag^{-1}$; this is a subsequence of $A$ of length $2n+2$. We have  $\pi_1(\sigma( A')) = \pi_1(\sigma( A )) - \pi_1( g ) \neq 0$. Thus, by Proposition \ref{prop_nonzero} we know that $A'$ has a zero-sum subsequence of length $2n$, and thus also $A$ has this zero-sum subsequence, completing the argument.
\end{proof}

\end{document}